\documentclass[11pt]{amsart}

\usepackage{amsmath,amsthm,amssymb}

\usepackage{hyperref}

\theoremstyle{plain}
\newtheorem{theorem}{Theorem}
\newtheorem{lemma}{Lemma}[section]
\newtheorem{proposition}[lemma]{Proposition}

\newtheorem{notation}[lemma]{Notation}
\newtheorem{claim}[lemma]{Claim}

\theoremstyle{definition}
\newtheorem{definition}[lemma]{Definition}

\DeclareMathOperator{\rng}{rng}

\begin{document}

\title[Degrees of primitive recursive $m$-reducibility]{A note on the degree structure of primitive recursive $m$-reducibility}

\author[B. S. Kalmurzayev]{Birzhan S. Kalmurzayev}
\address{Kazakh-British Technical University, 59 Tole Bi Street, Almaty, 050000, Kazakhstan}
\address{Al-Farabi Kazakh National University, 71 al Farabi Avenue, Almaty, 050040, Kazakhstan}
\email{birzhan.kalmurzayev@gmail.com}

\author[N. A. Bazhenov]{Nikolay A. Bazhenov}
\address{Kazakh-British Technical University, 59 Tole Bi Street, Almaty, 050000, Kazakhstan}
\email{n.bazhenov.1@gmail.com}

\author[A.M. Iskakov]{Alibek M. Iskakov}
\address{Kazakh-British Technical University, 59 Tole Bi Street, Almaty, 050000, Kazakhstan}
\address{Al-Farabi Kazakh National University, 71 al Farabi Avenue, Almaty, 050040, Kazakhstan}
\email{bheadr73@gmail.com}

\date{\today}

\begin{abstract}
	Let $\mathbf{C}^{pr}_m$ be the upper semilattice of degrees of computable sets with respect to primitive recursive $m$-reducibility. We prove that the first-order theory of $\mathbf{C}^{pr}_m$ is hereditarily undecidable.
\end{abstract}


\maketitle

\section{Introduction}

Let $A$ and $B$ be subsets of natural numbers. We say that $A$ is \emph{primitively recursively $m$-reducible} to $B$ (or $pr$-$m$-reducible, for short, and denoted $A \leq^{pr}_m B$) if there exists a primitive recursive function $f(x)$ such that 
\[
		\forall x (x \in A\ \Leftrightarrow\ f(x) \in B).
\]
By $\mathbf{C}^{pr}_{m}$ we denote the upper semilattice of all computable degrees of sets under primitive recursive $m$-reducibility. As usual, we do not include the degrees $\deg^{pr}_m(\emptyset)$ and $\deg^{pr}_m(\omega)$ into $\mathbf{C}^{pr}_{m}$. 

For the background on computability theory, the reader is referred to \cite{Odi-92}.
Notice that a standard proof (see, e.g., Proposition VI.1.8 in \cite{Odi-92}) gives the following:
\begin{lemma}
	The upper semilattice $\mathbf{C}^{pr}_{m}$ is distributive.
\end{lemma}

\medskip

Following the outline from \cite{DN-2000}, this note proves the following main result:
\begin{theorem}\label{theo:undecidable}
	The first-order theory of $\mathbf{C}^{pr}_{m}$ is hereditarily undecidable.
\end{theorem}

The proof consists of three sections. Note that here we tried to make our exposition as \emph{self-contained} as possible.

Section~\ref{sect:schema} briefly describes the general proof scheme. In Section~\ref{sect:super-sparse}, we define an appropriate modification of the notion of a \emph{super sparse} set taken from~\cite{Ambos-Spies-86,DN-2000}. We also establish some useful properties of the modified notion.

In the last section, we prove our main lemma (Lemma~\ref{lem:main-BA}) which allows us to interpret (with parameters) an appropriate structure $\mathcal{I} (\mathcal{B})$, having hereditarily undecidable theory, inside the semilattice $\mathbf{C}^{pr}_m$.

\subsection{Some preliminary facts}\label{subsect:prelim}

We use the list $(p_e)_{e\in\omega}$ which effectively (but not primitively recursively) lists all unary primitive recursive functions. We can also effectively list the primitive recursive time-bound functions $r_e(x) \geq p_e(x)$, where $r_e(x)$ gives the number of computational steps needed for $p_e(x)$  to converge. We note that for any partial computable function $\varphi_e(x)$, including the primitive recursive functions $p_{e}$, the stage $s$ approximation $\varphi_{e,s}(x)$ is primitive recursive (in $e$, $s$, and $x$).

In the constructions of some computable sets, we will use the following
\begin{lemma}[folklore]\label{lem:Graph}
	Let $f(x)$ be a strictly increasing, total computable function such that the graph of $f$ is a primitive recursive set. Then there is a primitive recursive procedure which given an element $y\in \omega$, computes the following:
	\begin{enumerate}
		\item whether $y$ belongs to the set $\rng(f)$;
		\item if $y\in \rng(f)$, then the procedure finds the number $x$ such that $f(x)=y$.
	\end{enumerate}
\end{lemma}
\begin{proof}
	Since $f$ is strictly increasing, it is clear that $f(y+1) \geq y+1 > y$. Thus, in order to check whether $y\in \rng(f)$, it is sufficient to search for the minimal $x \leq y$ such that $(x,y) \in \mathrm{graph}(f)$. If such an $x$ exists, then $y= f(x)$. Otherwise, $y\not\in \rng(f)$. Observe that the described procedure is primitive recursive.
\end{proof}


\section{The outline of the proof}\label{sect:schema}

For the preliminaries on countable Boolean algebras, we refer to, e.g., the monograph~\cite{Goncharov-Book}.

Let $n\geq 1$. Then a structure $\mathcal{B} = (\omega; \vee, \wedge, \overline{(\cdot)}; \approx)$ is called a \emph{$\Sigma^0_n$-Boolean algebra}, if:
\begin{itemize}
	\item the functions $\vee,\wedge,\overline{(\cdot)}$ are computable;
	
	\item the binary relation $\approx$ belongs to $\Sigma^0_n$;
	
	\item $\approx$ is a congruence of the structure $\mathcal{B}^{\#} = (\omega; \vee, \wedge, \overline{(\cdot)})$;
	
	\item the quotient structure $\mathcal{B}^{\#}/\!\approx$ is Boolean algebra
\end{itemize}
The (natural) preorder relation $\preceq$ on a $\Sigma^0_n$-Boolean algebra $\mathcal{B}$ is introduced as follows: $x\preceq y$ if and only if $x\vee y \approx y$. It is not hard to check that $(x\approx y)$ holds if and only if $(x \preceq y)$ and $(y \preceq x)$.

By $0_{\mathcal{B}}$ we denote the least element in the algebra $\mathcal{B}$ (here we identify the element $0_{\mathcal{B}}$ from $\omega$ and its image $[0_{\mathcal{B}}]_{\approx}$ in the quotient algebra $\mathcal{B}^ {\#}/\!\approx$).

\medskip

\begin{definition}[\cite{Nies-97}]
	A $\Sigma^0_n$-Boolean algebra $\mathcal{B}$ is called \emph{effectively dense} if there is total $\Delta^0_n$-computable function $F(x)$ such that
	\[
	\forall x [ x\not\approx 0_{\mathcal{B}}\ \rightarrow\  0_{\mathcal{B}} \prec F(x) \prec x].
	\]
\end{definition}

Let $\mathcal{C} = (C; \vee, \wedge, \overline{(\cdot)})$ be a Boolean algebra. Recall that a non-empty set $I\subseteq C$ is  an \emph{ideal} of the algebra $\mathcal{C}$ if:
\begin{enumerate}
	\item $(\forall x,y\in I)[x\vee y \in I]$;
	
	\item if $x\leq_{\mathcal{C}} y$ and $y\in I$, then $x\in I$.
\end{enumerate}

Let $\mathcal{B}$ be a $\Sigma^0_n$-Boolean algebra. A set $J\subseteq \omega$ is called a \emph{$\Sigma^0_n$-ideal} of the algebra $\mathcal{B}$ if $J\in \Sigma^0_n$ and $J = \{ x : [x]_{\approx} \in I\}$ for some ideal $I$ in the Boolean algebra $\mathcal{B}^{\#}/\!\approx$.
It is known that the set of all $\Sigma^0_n$-ideals of $\mathcal{B}$ forms a lattice under the following operations:
\begin{itemize}
	\item the set-theoretic intersection $J\cap K$,
	
	\item the operation $J \vee K = \{ x\vee y : x\in J,\ y \in K\}$.
\end{itemize}
This lattice is distributive.

\begin{notation}
	For a $\Sigma^0_n$-Boolean algebra $\mathcal{B}$, by $\mathcal{I}(\mathcal{B})$ we denote the lattice of all $\Sigma^0_n$-ideals of $\mathcal{B}$.
\end{notation}

The proof of Theorem~\ref{theo:undecidable} will use the following result:

\begin{theorem}[Nies~\cite{Nies-97}]\label{theo:Nies}
	Let $\mathcal{B}$ be an effectively dense $\Sigma^0_n$-Boolean algebra. Then the theory $Th(\mathcal{I}(\mathcal{B}))$ is hereditarily undecidable.\footnote{Note that Nies~\cite{Nies-00} proved that $Th(\mathcal{I}(\mathcal{B})) \equiv_m Th(\mathbb{N},+,\times)$.}
\end{theorem}

\medskip

The proof of Theorem~\ref{theo:undecidable} consists of two main stages.

Let $\mathbf{a}$ be the $pr$-$m$-degree of a computable, non-primitive recursive set $A$. We define the set $\mathcal{B}(\mathbf{a})$ of all complemented elements in the lower cone of $\mathbf{a}$:
\begin{equation}\label{equ:B(a)}
	\mathcal{B}(\mathbf{a})  = \{ \mathbf{b} \in \mathbf{C}^{pr}_m : \exists \mathbf{c} (\mathbf{b} \wedge \mathbf{c} = \mathbf{0} \ \&\   \mathbf{b} \vee \mathbf{c} = \mathbf{a})\}.
\end{equation}

The first stage (Section~\ref{sect:super-sparse}) establishes that for some appropriate sets $A$, the structure $\mathcal{B}(\mathbf{a})$ is an effectively dense $\Sigma^0_2$-Boolean algebra. 

The second stage (Section~\ref{sect:lemma-finish}) shows that the constructed set $A$ also allows us to interpret (with one parameter $\mathbf{a}$) the lattice $\mathcal{I}(\mathcal{B}(\mathbf{a}))$ inside the structure $\mathbf{C}^{pr}_m$. This fact and Theorem~\ref{theo:Nies} together imply (see, e.g., \cite{ELTT-65,BurSan-75}) that the theory $Th(\mathbf{C}^{pr}_m)$ is hereditarily undecidable.

\medskip

\section{Super pr-sparse sets and their properties}\label{sect:super-sparse}

The notion of a super sparse set (introduced in~\cite{Ambos-Spies-86}) plays an important role in the proofs of~\cite{DN-2000}. Here we give a modification of this notion, which is suitable for our primitive recursive setting.

\begin{definition}\label{def:super-sparse}
	Let $A\subseteq \omega$ be a computable set, and let $f(x)$ be a strictly increasing computable function. We say that the set $A$ is \emph{super $pr$-sparse via $f$} if:
	\begin{itemize}
		\item[(a)] the set $A$ is not primitive recursive;
		
		\item[(b)] the graph $\Gamma_f$ of the function $f$ is primitive recursive;
		
		\item[(c)] $A\subseteq \mathrm{range}(f)$ and the value $\chi_A(f(x))$ can be computed in at most $f(x+1)$ steps;
		
		\item[(d)] for each primitive recursive function $p_i(x)$, there is $n_i\in\omega$ such that 
		\[
			(\forall n \geq n_i)(f(n+1) > p_i(f(n))).
		\]
	\end{itemize}
\end{definition}

The main result of this section is the following: 
\begin{proposition}\label{prop:eff-dense-BA}
	If $A$ is a super $pr$-sparse set, then the structure $\mathcal{B}(\mathbf{a})$ from Eq.~(\ref{equ:B(a)}) is an effectively dense $\Sigma^0_2$-Boolean algebra.
\end{proposition}

\medskip

\subsection{Existence of super pr-sparse sets} 
Before proving Proposition~\ref{prop:eff-dense-BA}, we establish the following necessary result.

\begin{lemma}\label{lem:super-sparse-exist}
	There exists a super $pr$-sparse set.
\end{lemma}
\begin{proof}
	We construct (stage-by-stage) a strictly increasing computable function $f(x)$ and a computable set $A \subseteq \mathrm{range}(f)$.
	
	At stage $0$ we set $f(0) = 0$. 	
We assume that the following objects are defined by the beginning of stage $s+1$:
	\begin{itemize}
		\item the values $f(0) < f(1) < \dots < f(s)$;
		
		\item the value $\chi_A(x)$, for each $x < f(s)$.
	\end{itemize}
	In addition, the following conditions are satisfied:
	\begin{itemize}
		\item[($\ast$)]  $A \cap \{ x : x \leq f(s)-1\} \subseteq \mathrm{range}(f)$;
		
		\item[($\ast\ast$)] $f(t) > p_j (f(t-1))$ for all $0\leq j<t \leq s$.
	\end{itemize}
	
	Our main goal at the stage $s+1$ is to appropriately define $\chi_A(f(s))$ and $f(s+1)$.
	To do this, we successively compute the following values:
	\begin{itemize}
		\item[(a)] $p_0(f(s)), p_1(f(s)), \dots, p_s(f(s))$;
		
		\item[(b)] $z = 1 + \max(f(s),p_0(f(s)),\dots,p_s(f(s)))$;
		
		\item[(c)] the number of steps $N$ which is necessary for successively computing $f(0),f(1),\dots,f(s)$ and all the values from~(a)--(b).
	\end{itemize}
	Without loss of generality, we may assume that $N \geq z$.
	
	While the values~(a)--(c) are being computed (in general, these computation processes are `slow' for us), for the numbers $y \in \{ f(s)+1,f(s)+2, f(s)+3,\dots\}$, one-by-one, we `quickly' declare the following:
	\[
		y\not\in \mathrm{range}(f) \text{ and } y \not\in A.
	\]
	Such a declaration ensures that the graph of the constructed function $f$ will be primitive recursive.
	
	After the required values have been computed, we take the least $w \geq N$ such that the values $\chi_{\mathrm{range}(f)}(w)$ and $\chi_A(w)$ have not been defined yet. We set $f(s+1) = w$ and
	\begin{equation}\label{equ:chi_a-aux}
		\chi_A(f(s)) = \begin{cases}
			0, & \text{if } p_s(f(s)) \geq 1,\\
			1, & \text{if } p_s(f(s)) = 0.
		\end{cases}
	\end{equation}	
	
	It is clear that $f(s+1) \geq z > p_j(f(s))$ for all $j \leq s$. Therefore, the corresponding conditions~($\ast$) and $(\ast\ast)$ will remain true at the beginning of the next stage $s+2$. In addition, the value $\chi_A(f(s))$ can be computed in at most $N \leq f(s+1)$ steps.
	
	The stage $s+1$ is finished. This concludes the construction.
	
	Eq.~(\ref{equ:chi_a-aux}) guarantees that the constructed set $A$ is not primitive recursive. Condition~($\ast\ast$) implies that the function $f$ satisfies the last condition of Definition~\ref{def:super-sparse}. The other conditions of Definition~\ref{def:super-sparse} can be easily deduced from an  analysis of the construction. Lemma~\ref{lem:super-sparse-exist} is proved.
\end{proof}

\medskip

\subsection{Proof of Proposition~\ref{prop:eff-dense-BA}}\label{subsect:01}

We prove that for a super $pr$-sparse set $A$, the partial order $\mathcal{B}(\mathbf{a})$ is (effectively) isomorphic to the quotient structure $\mathcal{C}^{\#}/ \!\approx$, where $\mathcal{C} = (\mathcal{C}^{\#};\approx)$ is the $\Sigma^0_2$-Boolean algebra given below. Thus, the structure $\mathcal{B}(\mathbf{a})$ can be identified with the $\Sigma^0_2$-algebra $\mathcal{C}$. After that, it will be sufficient to check that the $\Sigma^0_2$-algebra $\mathcal{C}$ is effectively dense.

\begin{notation}
	(1)\ The auxiliary Boolean algebra $\mathcal{C}^{\#}$ consists of all sets of the form $A\cap X$, where $X$ is a primitive recursive set. The Boolean operations $\vee, \wedge, \overline{(\cdot )}$ on $\mathcal{C}^{\#}$ are defined as follows:
	\begin{itemize}
		\item $(A\cap X) \vee (A \cap Y) = A\cap (X\cup Y)$; 
		
		\item $(A\cap X) \wedge (A\cap Y) = A\cap (X\cap Y)$;
		
		\item the complement of $A\cap X$ (inside the algebra $\mathcal{C}^{\#}$) is equal to $A\cap (\omega \setminus X)$.
	\end{itemize}
	Fix a (standard) effective list $(X_e)_{e\in\omega}$ of all primitive recursive sets. By identifying the set $A\cap X_e$ with the index $e\in \omega$, we may  assume that the domain of $\mathcal{C}^{\#}$ is equal to $\omega$.
	
	(2)\  The $\Sigma^0_2$-Boolean algebra $\mathcal{C} = (\omega; \vee,\wedge,\overline{(\cdot)}; \approx)$ is defined as follows:
	\begin{itemize}
		\item the structure $\mathcal{C}^{\#} = (\omega; \vee,\wedge,\overline{(\cdot)} )$ is given above;
		
		\item $i \approx j$ if and only if  the symmetric difference $(A\cap X_i) \triangle (A\cap X_j)$ is a primitive recursive set.
	\end{itemize}
\end{notation}

Note that indeed, $\mathcal{C}$ is a $\Sigma^0_2$-Boolean algebra:
\begin{itemize}
	\item There is a computable function $f_{\vee}(x,y)$ such that $X_i \cup X_j = X_{f_{\vee}(i,j)}$. A similar fact is true for $\wedge$ and $\overline{(\cdot)}$.
	
	\item The condition $(i \approx j)$ is equivalent to the following $\Sigma^0_2$-formula:
	\[
	\exists k \forall x [ \chi_{X_k}(x) = \chi_A(x) \cdot \max (\chi_{X_i}(x) \cdot (1 - \chi_{X_j}(x)), \chi_{X_j}(x) \cdot (1 - \chi_{X_i}(x)))].
	\]
\end{itemize}

\medskip

In order to establish an isomorphism between $\mathcal{B}(\mathbf{a})$ and $\mathcal{C}$, it is sufficient to prove the following result.

\begin{lemma}\label{lem:BA-01}
	Suppose that a set $A$ is super $pr$-sparse via a function $f$.
	
	(i)\ Let $\mathbf{b}$ and $\mathbf{c}$ be degrees from $\mathcal{B}(\mathbf{a})$ such that 
	\begin{equation}\label{equ:complemented-pair}
		\mathbf{b} \wedge \mathbf{c} = \mathbf{0}\ \text{and}\ \mathbf{b} \vee \mathbf{c} = \mathbf{a}.
	\end{equation}
	Then there exists a primitive recursive set $X$ such that $A\cap X \in \mathbf{b}$ and $A\cap \overline{X} \in \mathbf{c}$.
	
	(ii)\ Let $X$ be a primitive recursive set. Then Condition~(\ref{equ:complemented-pair}) holds for the degrees $\mathbf{b} = \deg^{pr}_m(A\cap X)$ and $\mathbf{c} = \deg^{pr}_m(A\cap \overline{X})$.
	
	(iii)\ Let $X,Y$ be primitive recursive sets. Then the condition $(A\cap X) \leq^{pr}_m (A\cap Y)$ holds if and only if:
	\begin{itemize}
		\item either $A\cap Y \neq \emptyset$ and the set $A\cap(X \setminus Y)$ is primitive recursive,
		
		\item or $A\cap Y = A\cap X = \emptyset$.
	\end{itemize}
	In particular, the properties~(i)--(iii) imply that the map $e\mapsto \deg_{m}^{pr}(A\cap X_e)$ induces an isomorphism from $\mathcal{C}$ onto $\mathcal{ B}(\mathbf{a})$.
\end{lemma}

Before proving Lemma~\ref{lem:BA-01}, we establish some simple properties of $pr$-$m$-reducibility:

\begin{proposition}\label{claim:001}
	Let $B\subsetneq \omega$, and let $X$ be a primitive recursive set. Then
	\begin{itemize}
		\item[(a)] $B\cap X \leq^{pr}_m B$;
		
		\item[(b)] $B \equiv^{pr}_m (B\cap X) \oplus (B\cap \overline{X})$.
	\end{itemize} 
\end{proposition}
\begin{proof}
	(a)\ We choose an element $c\not\in B$. The reduction $g\colon B\cap X \leq^{pr}_m B$ is defined as follows:
	\[
	g(x) = \begin{cases}
		x, & \text{if } x\in X,\\
		c, & \text{if } x\not\in X.
	\end{cases}
	\]
	
	(b)\ By item~(a), it  is sufficient to construct a reduction $h\colon B\leq^{pr}_m (B\cap X) \oplus (B\cap \overline{X})$. We define
	\[
	h(x) = \begin{cases}
		2x, & \text{if } x\in X,\\
		2x+1, & \text{if } x\not\in X.
	\end{cases}
	\]
\end{proof}

\medskip

\begin{proof}[Proof of Lemma~\ref{lem:BA-01}]
	\textbf{(i)}\ We choose computable sets $B \in \mathbf{b}$ and $C\in\mathbf{c}$. If one of the sets $B$ or $C$ is primitive recursive, then the desired fact is obvious. (For example, if $B$ is primitive recursive, then we choose $X=\{ a\}$ for some element $a\in A$.) Thus, we may assume that $B$ and $C$ are not primitive recursive.
	
	Since $\mathbf{b} \vee \mathbf{c} = \mathbf{a}$, there is a reduction $g\colon A\leq^{pr}_m B\oplus C$. We define a primitive recursive set
	\[
		X = \{ x : g(x) \text{ is even}\}.
	\]
	It is clear that the function $g_1(x) = \lfloor g(x)/2\rfloor$ provides two reductions: $A\cap X \leq^{pr}_m B$ and $A\cap \overline{X} \leq^{pr}_m C$. 
	
	We show that $B\leq^{pr}_m A\cap X$ (the proof for $C\leq^{pr}_m A\cap \overline{X}$ is similar). Since $B\leq^{pr}_m A\equiv^{pr}_m (A\cap X) \oplus (A\cap \overline{X})$ (see Proposition~\ref{claim:001}), we obtain that $B\equiv^{pr}_m B_1 \oplus B_2$, where $B_1\leq^{pr}_m A\cap X$ and $B_2 \leq^{pr}_m A\cap \overline{X}$. 
	
	The fact that $B_2 \leq^{pr}_m B$ and $B_2 \leq^{pr}_m A\cap\overline{X} \leq^{pr}_m C$ implies that the set $B_2$ is primitive recursive (recall that $\mathbf{b} \wedge \mathbf{c} = \mathbf{0}$). Therefore, we have $B\equiv^{pr}_m B_1 \oplus B_2 \equiv^{pr}_m B_1 \leq^{pr}_m A\cap X$.
	Hence, we conclude that $B\equiv_{m}^{pr} A\cap X$ and $C\equiv_{m}^{pr} A\cap \overline{X}$.

	\medskip

\textbf{(ii)}\ Let $X$ be a primitive recursive set. Proposition~\ref{claim:001} implies that $A\equiv^{pr}_m (A\cap X) \oplus (A\cap \overline{X})$, that is $\mathbf{a} = \mathbf{b} \vee \mathbf{c}$. Recall that the set $A$ is super $pr$-sparse via the function $f(x)$.

Now let $D$ be an arbitrary set such that $p \colon D\leq^{pr}_m A\cap X$ and $q \colon D \leq^{pr}_m A\cap \overline{X}$. We describe a primitive recursive algorithm for computing the characteristic function $\chi_D(x)$.

1)\ Compute the values $p(x)$ and $q(x)$. By Lemma~\ref{lem:Graph}, the set $\mathrm{range}(f)$ is primitive recursive. If $p(x) \not\in \mathrm{range}(f)$ or $q(x) \not\in \mathrm{range}(f)$, then it is clear that $\{ p(x), q(x) \} \not\subseteq A$ and thus, $x\not\in D$.

If $p(x) \not\in X$ or $q(x)\not\in \overline{X}$, then in a similar way, we can obtain that $x\not\in D$.

2)\ Hence, we may assume that $p(x) \in \mathrm{range}(f) \cap X$ and $q(x) \in \mathrm{range}(f) \cap \overline{X}$ (in particular, $p(x) \neq q(x)$). By Lemma~\ref{lem:Graph}, we can find (via a primitive recursive algorithm) the numbers $k$ and $l$ such that $p(x) = f(k)$ and $q(x) = f(l)$. Without loss of generality, we assume that $k < l$ (the case when $k > l$ is treated similarly).

3)\ By Property~(c) of Definition~\ref{def:super-sparse}, one can find the value $\chi_A(f(k))$ in $q(x) = f(l)$ computation steps. Then the fact that $p\colon D\leq^{pr}_m A\cap X$ implies that $\chi_D(x) = \chi_A(f(k)) \cdot \chi_X(f(k))$.

The algorithm described above shows that the set $D$ is primitive recursive. Therefore, we have $\mathbf{b} \wedge \mathbf{c} = \mathbf{0}$.

\medskip

\textbf{(iii)}\ For the case when $A\cap Y = \emptyset$, the desired property is clearly satisfied. Thus, we may assume that $A\cap Y \neq \emptyset$.

$(\Leftarrow)$. Suppose that the set $A\cap(X\setminus Y)$ is primitive recursive. We choose elements $c\in A\cap Y$ and $d\not\in A$. Then the function
\[
h(x) = \begin{cases}
	c, & \text{if } x\in A\cap (X\setminus Y),\\
	d, & \text{if } x \not\in X,\\
	x, & \text{otherwise},
\end{cases}
\]
provides a reduction $A\cap X \leq^{pr}_m A\cap Y$. In order to show that $h$ is indeed a reduction, it is sufficient to notice the following: if $h(x)  \not\in\{ c,d\}$, then we have either $x \in A\cap X \cap Y$ or $x\not\in A$.

$(\Rightarrow)$. Let $A\cap X \leq^{pr}_m A\cap Y$. By Proposition~\ref{claim:001}, we have $A\cap X \equiv^{pr}_m (A\cap X \cap Y) \oplus (A\cap (X\setminus Y))$. Similarly,
\begin{equation}\label{equ:small-aux-001}
	A\cap \overline{Y} \equiv^{pr}_m (A\cap (X\setminus Y)) \oplus (A\cap \overline{Y} \cap \overline{X}).
\end{equation}

By item~(ii), the infimum of $\deg^{pr}_m(A\cap \overline{Y})$ and $\deg^{pr}_m(A\cap Y)$ is equal to $\mathbf {0}$. By Eq.~(\ref{equ:small-aux-001}), we obtain that the infimum for $\deg^{pr}_m(A\cap (X\setminus Y))$ and $\deg^{pr}_m(A\cap Y)$ is also equal to $\mathbf{0}$. Now, since $A\cap(X\setminus Y) \leq^{pr}_m A\cap X \leq^{pr}_m A\cap Y$, we deduce that $A\cap (X\setminus Y)$ is primitive recursive. Lemma~\ref{lem:BA-01} is proved.
\end{proof}

\bigskip

Now, in order to prove Proposition~\ref{prop:eff-dense-BA}, it remains to establish the following result.

\begin{lemma}\label{lem:BA-02}
	The $\Sigma^0_2$-Boolean algebra $\mathcal{B}(\mathbf{a})$ is effectively dense.
\end{lemma}
\begin{proof}
For a given primitive recursive set $X =X_e$, we construct a new primitive recursive set $Y = X_{F(e)} \subset X$ with the following properties: if $A\cap X$ is not primitive recursive, then $A\cap Y$ is also non-primitive recursive and $A\cap Y <^{pr}_m A\cap X$. 
We also establish that (in this construction) the function $e\mapsto F(e)$ is $\mathbf{0}'$-computable: this gives us precisely the notion of effective density for our $\Sigma^0_2$-Boolean algebra $\mathcal{B}(\mathbf{a})$.

For a primitive recursive set $X$, we satisfy the following requirements:
\begin{itemize}
	\item[$R_{i}$:] if there are infinitely many $x$ such that $\chi_{A\cap X}(x) \neq p_i(x)$, then there exist $x'$ and $x''$ such that $\chi_{A\cap Y}(x') \neq p_i(x')$ and $\chi_{A\cap (X\setminus Y)}(x'') \neq p_i(x'')$. 
\end{itemize}
If the set $A\cap X$ is not primitive recursive, then $\exists^{\infty} x(\chi_{A\cap X}(x) \neq p_i(x))$. Hence, for such $A\cap X$ the requirements  $R_i$  ensure the following:
\begin{itemize}
	\item the set $A\cap Y$ is not primitive recursive, 
	
	\item the set $A\cap(X\setminus Y )$ is not primitive recursive~--- by Lemma~\ref{lem:BA-01}.(iii), this guarantees that $A\cap X \nleq^{pr}_m A\cap Y$.
\end{itemize}	
In addition, in the construction we will build a primitive recursive function $h \colon A\cap Y \leq^{pr}_m A\cap X$. This will imply all the desired properties for $A\cap Y$.

We fix an element $d\not\in A$. 

\medskip

\emph{The construction} consists of $R_i$-stages for $i\in\omega$ (the stages are carried out one-by-one, in ascending order of the index $i$).

Let $m_0$ be the least number such that $\chi_Y(m_0)$ has not been defined by the beginning of the $R_i$-stage. For the numbers $x\in \{m_0,m_0+1,m_0+2,\dots\}$, we successively `promptly' define $\chi_Y(x) := \chi_X(x)$ and $h(x) := x$. In other words, the constructed set $Y$ promptly `copies' the set $X$.

At the same time, we run the computation processes (that are `slow' for us) for the following values: 
\[
p_i(m_0),\chi_{A\cap X}(m_0),p_i(m_0+1),\chi_{A\cap X}(m_0+1),p_i(m_0+2), \chi_{A\cap X} (m_0+2),\dots.
\]

The procedures described above continue, until one of the following two cases is met: 
\begin{itemize}
	\item[(a)] we have found a (least) $y\geq m_0$ such that $\chi_{A\cap X}(y) =  0$ and $p_i(y)\neq 0$;
	
	\item[(b)] we have found $z\geq m_0$ such that $\chi_{A\cap X}(z) = 1$ and $p_i(z) \neq 1$.
\end{itemize}

\emph{Case~(a).} In this case, we have $y\not\in A\cap X$; hence, $y\not\in A\cap (X\setminus Y)$. Since the set $Y$ had copied $X$, we also get that $y\not\in A\cap Y$. Since $p_i(y) \neq 0$, we conclude that we can choose $x'=x''=y$ for the requirement $R_i$. Our $R_i$ is forever satisfied, and we proceed to the $R_{i+1}$-stage.

\emph{Case~(b).} Since $Y$ had copied $X$, we deduce that we have found $z$ such that $\chi_{A\cap Y}(z) = 1 \neq p_i(z)$. Let $m_1$ be the least number such that $\chi_Y(m_1)$ has not been defined by the moment when this $z$ was found.

Then for the numbers $w\in\{m_1,m_1+1,m_1+2,\dots\}$, we successively define $\chi_Y(w) = 0$ and $h(w) = d$. At the same time, we run computations for 
\[
	p_i(m_1),\chi_{A\cap X}(m_1),p_i(m_1+1),\chi_{A\cap X}(m_1+1),\dots.
\]

We continue this procedure until we find a number $z_1 \geq m_1$ such that $\chi_{A\cap X}(z_1) \neq p_i(z_1)$. Then, since $z_1 \not\in Y$, we get that $\chi_{A\cap (X\setminus Y)}(z_1) = \chi_{A\cap X}(z_1) \neq p_i (z_1)$. The requirement $R_i$ is forever satisfied (for $x'=z$ and $x''=z_1$), and we proceed to the $R_{i+1}$-stage.

The construction is described.

\medskip

\emph{Verification.} If the set $A\cap X$ is not primitive recursive, then each $R_i$-requirement will be satisfied with the help of the $R_i$-stage. Indeed, there are infinitely many $x$ such that $\chi_{A\cap X}(x) \neq p_i(x)$. Therefore, at the $R_i$-stage, one of Cases~(a) or~(b) will be eventually satisfied. In Case~(b) we will also eventually find a suitable number $z_1$.

Notice that for the case when $A\cap X$ is primitive recursive, the construction may `get stuck' (forever) at some $R_i$-stage. Nevertheless, it is easy to show that the constructed function $h$ always provides a reduction $A\cap Y \leq^{pr}_m A\cap X$.

We show that the function $e\mapsto F(e)$ (where $Y=X_{F(e)}$) is $\mathbf{0}'$-computable. In order to (uniformly) find the index $F(e)$, it is sufficient  to check for a given $i\in\omega$, whether the $R_i$-stage for our $X = X_e$ eventually stops. It is clear that this checking procedure is effective in the oracle $\mathbf{0}'$. Hence, the function $F$ is $\mathbf{0}'$-computable.

Lemma~\ref{lem:BA-02} and Proposition~\ref{prop:eff-dense-BA} are proved.
\end{proof}

We note that the proof of Proposition~\ref{prop:eff-dense-BA} never uses the item~(d) of Definition~\ref{def:super-sparse}.

\medskip

\section{The main lemma}\label{sect:lemma-finish} 

As in the proof of \cite[Theorem 3.2]{DN-2000}, it is sufficient to prove the following result:

\begin{lemma}\label{lem:main-BA}
	Suppose that a set $A$ is super $pr$-sparse via the function $f$. Then for every $\Sigma_{2}^{0}$-ideal $I$ of the algebra $\mathcal{B}(\mathbf{a})$, there exists a degree $\mathbf{c}_{I}\leq \mathbf{a}$ such that
	$$
		(\forall x\in \mathcal{B}(\mathbf{a}))[\mathbf{x}\in I\Leftrightarrow \mathbf{x}\leq \mathbf{c}_{I}].
	$$
\end{lemma}

Then, indeed, the lattice $\mathcal{I}(\mathcal{B}(\mathbf{a}))$ has the following interpretation (with parameter $\mathbf{a}$) inside the structure $\mathbf{C}^{pr}_m$:
\begin{itemize}
	\item the formula $\varphi_{dom}(x) = (x \leq\mathbf{a})$ defines the domain of the lattice;
	\item the formula
	\[
	\varphi_{\leq}(x,y) = \forall z [ \exists w (z\vee w = \mathbf{a} \,\&\, z\wedge w = \mathbf{0}) \rightarrow (z \leq x \rightarrow z\leq y)]
	\]
	defines the ordering on $\mathcal{I}(\mathcal{B}(\mathbf{a}))$
\end{itemize}
(See more details in \cite{DN-2000}.) Theorem~\ref{theo:Nies} and Proposition~\ref{prop:eff-dense-BA} together imply that the theory $Th(\mathbf{C}^{pr}_m)$ is hereditarily undecidable.

\medskip

\begin{proof}[Proof of Lemma~\ref{lem:main-BA}]
	Recall that $(X_{i})_{i\in\omega}$ is the list of all primitive recursive sets (see Subsection~\ref{subsect:01}). Let $\psi(x)$ be a total $\Delta_{2}^{0}$-computable function such that $\mathrm{range}(\psi) = \{ i : \deg^{pr}_m(A\cap X_i) \in I\}$. Such a function $\psi$ exists, since $I$ is a $\Sigma_{2}^{0}$-ideal. We fix a primitive recursive function $\psi_{1}(x,s)$ such that $\psi(x) = \lim_s \psi_1(x,s)$.
	
	As in \cite[Lemma~3.12]{DN-2000}, we satisfy the following requirements:
	\begin{itemize}
		\item[$R_{e}$:] $A\cap X_{\psi(e)} \leq_{m}^{pr} C_{I}$;
		
		\item[$H_{\langle i,j\rangle}\colon$] if $p_j\colon A\cap X_i \leq^{pr}_m C_I$, then $A\cap X_i \leq^{pr}_m \bigoplus_{m\leq \langle i,j\rangle} (A\cap X_{\psi(m)})$.
	\end{itemize}
	The $R_{e}$-requirements ensure that $\mathrm{deg}_{m}^{pr}(C_{I})$ is an upper bound for the ideal $I$. The $H_{\langle i,j\rangle}$-requirements guarantee that for every degree $\deg_{m}^{pr}(A\cap X_{i})$ below $C_{I}$, there is a finite collection of elements from $I$ computing this degree.
	
	We construct the set $C_{I}$ as follows: we build a primitive recursive function $g(x)$, and we set $C_{I} := g^{-1}(A)$. It is clear that in this case $g\colon C_{I}\leq_{m}^{pr} A$.
	
	\medskip
	
	For $e\in\omega$, we define an auxiliary primitive recursive function $r_{e}$ as in Subsection~\ref{subsect:prelim}. The value $r_{e}(x)$ is equal to the number of steps needed to compute $p_{e}(x)$. Without loss of generality, we may assume that $p_{e}(x)\leq r_{e}(x)$ and $r_{e}(x)\leq r_{e+1}(x)$ for all $e,x\in\omega$.
	
	Let $r_{-1}(x) = 0$. Without loss of generality, we assume that $f(0)=0$ and $0\notin A\cup C_{I}$. Recall that
	\[
	\langle i, j\rangle = 2^{i} \cdot (2j+1).
	\]
	In particular, for a fixed $i\in\omega$, the distance between a pair of `neighbours' of the form $\langle i,\cdot \rangle$ is equal to $\langle i, j+1 \rangle - \langle i,j \rangle = 2^{i+1}$.
	
	Now we describe the construction.
	
	\medskip
	
	A \textit{requirement} $R_{e}$ is satisfied as follows. For each number $w\in\text{range}(f)$, we define the finite set of $R_{e}$-\textit{coding locations} associated with $w$. We say that a number $n$ is an $R_{e}$-coding location for the number $f(m)$ if the following conditions hold:
	\begin{itemize}
		\item[(i)] $n = \langle e,l \rangle$ for some $l \geq e$;
		
		\item[(ii)] $f(m)\leq n < f(m+1)$;
		
		\item[(iii)] $n$ is strictly greater that the number of steps needed to compute $r_{e-1}(f(m))$; in particular, this implies that $n > r_{e-1}(f(m))$.
	\end{itemize}
	Intuitively speaking, the $R_{e}$-coding locations are used to encode the set $A\cap X_{\psi(e)}$ inside our $C_{I}$. In addition, the $R_{e}$-coding locations will help us to build $g\colon C_{I} \leq_{m}^{pr} A$.
	
	Similar to Lemma 3.12 of \cite{DN-2000}, Condition (iii) will be important for satisfying the $H_{\langle i,j\rangle}$-requirements (see below). Informally speaking, to satisfy the $R_{e}$-requirements, it suffices to work only with Conditions~(i) and (ii).
	
	\begin{claim}\label{claim:002-a}
		\begin{enumerate}[(a)]
			\item Checking whether $n$ is an $R_{e}$-coding location for a number $f(m)$ is a primitive recursive procedure (in $e,m,n$).
			\item For sufficiently large $m$, the set of $R_{e}$-coding locations for $f(m)$ is non-empty.
		\end{enumerate}
	\end{claim}
	\begin{proof}
		(a) One can primitively recursively check Condition~(i). By Lemma~\ref{lem:Graph}, one can primitively recursively find the largest number $v$ such that $v \leq n$ and $v\in \mathrm{range}(f)$. Moreover, we can promptly find the number $k$ such that $f(k) = v$. Condition~(ii) holds if and only if $k=m$. 
		
		After that, to check~(iii), we take precisely $(n-1)$ steps in the computation process for $r_{e-1}(v)$. The number $n$ is an appropriate $R_e$-coding location if and only if the value $r_{e-1}(v)$ has been computed in at most $(n-1)$ steps.
		
		(b)\ We may assume that $m > \langle e,e\rangle$. We define a primitive recursive function $q$ as follows: $q(x)$ equals to $3\cdot 2^{e+1} + M(x)$, where $M(x)$ is equal to the number of steps needed to compute $r_{e-1}(x)$.
		
		By item (d) of Definition~\ref{def:super-sparse}, for a sufficiently large $m$, we always have $f(m+1) > q(f(m))$. Then the interval $\left[M(f(m)) +1, f(m+1)\right)$ (whose length is $> 2^{e+1} + 1$) will definitely have an appropriate $R_e$-coding location of the form $\langle e,l\rangle$.
	\end{proof}
	
	\medskip
	
	Recall that $0\notin A$. We define a function $g(x)$ (and the corresponding set $C_{I}$) as follows.
	\begin{enumerate}
		\item For a given $x$, we find the greatest $m$ such that $f(m)\leq x$. By Lemma~\ref{lem:Graph}, this procedure is primitive recursive. We also find $e^{\prime}$ and $l$ such that $n = \langle e^{\prime}, l \rangle$.
		
		\item If $x$ is not an $R_{e^{\prime}}$-coding location for the number $f(m)$, then we set $g(x) = 0$ and $x\notin C_{I}$.
		
		\item Let $x$ be an $R_{e^{\prime}}$-coding location for $f(m)$. We compute the value $u = \psi_{1}(e^{\prime},x)$. We proceed with $x$ computational steps for the value $\chi_{X_{u}}(f(m))$.
		
		If the value is computed in $x$ steps and $f(m) \in X_u$, then we set $g(x) = f(m)$ and $\chi_{C_I}(x) = \chi_A(f(m ))$.
		Otherwise, we define $g(x) = 0$ and $x\not\in C_I$.
	\end{enumerate}
	Claim~\ref{claim:002-a} implies that the function $g(x)$ is primitive recursive.
	
	\medskip
	
	We show that the choice of our $R_e$-coding locations allows us to `encode' $A\cap X_{\psi(e)}$ inside the constructed $C_I$.
	
	\begin{claim}\label{claim:003-a}
		There is a primitive recursive function $h(y)$ with the following property:
		\[
		\exists y_0 (\forall y \geq y_0)[ y\in A\cap X_{\psi(e)}\ \leftrightarrow\ h(y) \in C_I].
		\]
		In particular, this implies $A\cap X_{\psi(e)}\leq_{m}^{pr} C_{I}$.
	\end{claim}
	\begin{proof}
		We define an auxiliary primitive recursive function $L(y)$: the value $L(y)$ is equal to the number of steps needed to compute the values $\chi_{X_{\psi(e)}}(y)$ and $r_{e-1}(y)$.
		
		Choose a number $y_{0}$ with the following three properties:
		\begin{itemize}
			\item $\psi_{1}(e,y) = \psi(e)$ for all $y\geq y_{0}$;
			
			\item $y_{0} = f(m_{0})$, and for every $m\geq m_{0}$, the set of $R_{e}$-coding locations for $f(m)$ is non-empty;
			
			\item if $m\geq m_{0}$, then $f(m+1)$ is strictly greater than $\langle e, L(f(m)) \rangle$ (the existence of such $m_{0}$ is guaranteed by item (d) of Definition~\ref{def:super-sparse}).
		\end{itemize}
		
		Recall our assumption that $0\notin C_{I}$. For $y\geq y_{0}$, we define the value $h(y)$ via the following rules:
		\begin{enumerate}
			\item If $y\notin \text{range}(f)$ or $y\notin X_{\psi(e)}$, then it is clear that $y\notin A\cap X_{\psi(e)}$. Set $h(y) = 0$ for such $y$.
			\item If $y\in \text{range}(f)\cap X_{\psi(e)}$, then define $h(y) = \langle e,L(y) \rangle$.
		\end{enumerate}
		
		It is clear that the function $h$ is primitive recursive. Moreover, for Case~(2), our choice of $y_0$ ensures the following. Let $m$ be the number such that $y = f(m) \geq y_0$. In the definition of the value $g(x)$ for the element $x :=\langle e, L(y)\rangle$, the construction of $g$ (described above) ensures  that
		\[
			y = f(m) \in A\ \Leftrightarrow\ x\in C_I.
		\]
		Therefore, we obtain the equivalence $y\in A\cap X_{\psi(e)}$ $\Leftrightarrow$ $h(y) \in C_I$.
	\end{proof}
	
	Claim~\ref{claim:003-a} shows that every $R_{e}$-requirement is met.
	
	\medskip
	
	\textit{$H_{\langle i,j \rangle}$-requirements}. We show that Conditions~(i)-(iii) of the definition of $R_{e}$-coding locations allow us to satisfy the $H_{\langle i,j \rangle}$-requirements.
	
	\begin{claim}
		Every $H_{\langle i,j \rangle}$-requirement is met.
	\end{claim}
	\begin{proof}
		Let $p_{j}$ be a $pr$-$m$-reduction from $A\cap X_{i}$ to $C_{I}$. We build a primitive recursive function
		\begin{equation}\label{equ:aux-last}
			h\colon A\cap X_i \leq^{pr}_m \bigoplus_{m\leq \langle i,j\rangle} (A\cap X_{\psi(m)}).
		\end{equation}
		Similar to Claim~\ref{claim:003-a}, it is sufficient to define the value $h(x)$ only for sufficiently large $x$. In particular, we may assume that $\psi_{1}(m,x) = \psi(m)$ for every $m\leq \langle i,j \rangle$. Recall that $0\notin A\cup C_{I}$.
		
		If $x\notin\mathrm{range}(f)$ or $x\notin X_{i}$, then it is clear that $x\notin A\cap X_{i}$, and one can set $h(x) = 0$. Therefore, we assume that $x\in X_i$ and $x = f(m)$ for some (sufficiently large) $m$.
		
		We compute the values $y:= p_j(x)$ and $g(y)$. If $g(y) = 0$ or $g(y) \not\in\mathrm{range}(f)$, then $g(y)\not\in A$ and $y\not\in C_I$ (hence, we define $h(x) :=0$). Thus, we may assume that $g(y) = f(k)$ for some $k>0$ (recall that $f(0) = 0$).
		
		\smallskip
		
		\emph{Case~1.} Suppose that $k<m$. Then by item~(c) of Definition~\ref{def:super-sparse}, the value $\chi_A(g(y)) = \chi_{A}(f(k))$ can be computed in at most $f(m ) = x$ steps. We have
		\[
		x \in A\cap X_i\ \Leftrightarrow\ y =p_j(x)\in C_I\ \Leftrightarrow\ f(k) = g(y) \in A.
		\]
		We promptly compute $\chi_A(g(y))$ in $x$ steps, and we set: $h(x) :=0$ if $g(y) \not\in A$; and $h(x) := a_0$ if $g(y) \in A$. Here $a_0$ is some fixed element from $A\cap X_{\psi(0)}$.
		
		\smallskip
		
		\emph{Case~2.} Suppose that $k\geq m$. Since $g(y)$ is equal to $f(k)$, the number $y$ is an $R_e$-coding location for $f(k)$, for some $e\in\omega$. Hence, $y \geq f(k) \geq f(m) = x$.
		
		By item~(d) of Definition~\ref{def:super-sparse}, we may further assume that we work with sufficiently large $x$ such that $f(l+1) > r_j(f(l))$ for $ l\geq m$. This implies that $f(m+1) > r_j(f(m)) \geq p_j(f(m)) = p_j(x) = y$.
		
		We obtain $f(m) \leq y < f(m+1)$, and hence, $k = m$.
		
		Assume that $e > \langle i,j\rangle$. But then each $R_e$-coding location $z$ for $f(m)$ must satisfy $z > r_{e-1}(f(m)) = r_{e-1}(x) \geq r_{ \langle i,j\rangle}(x) \geq r_j(x) \geq p_j(x) = y$, and the number $y$ cannot be an $R_e$-coding location. We deduce that $e \leq \langle i,j\rangle$.
		
		We know that $x\in A\cap X_i\ \Leftrightarrow\ y = p_j(x)\in C_I$.
		If the value $\chi_{X_{\psi(e)}}(g(y)) = \chi_{X_{\psi(e)}}(f(m))$ can be computed in $\leq y$ steps, then we have
		\[
		y\in C_I \ \Leftrightarrow\ g(y) \in A\cap X_{\psi(e)}, \text{ where } e \leq \langle i,j\rangle.
		\]
		Using this fact, we can define $h(x)$ in a suitable manner (as the number from the $e$-th component of the set $\bigoplus_{m} (A\cap X_{\psi(m)})$ corresponding to $g (y)$). If $\chi_{X_{\psi(e)}}(g(y))$ cannot be computed in $\leq y$ steps, then $y\not\in C_I$ and $h(x) :=0$.
		
		\medskip
		
		A simple analysis of the described construction shows that the desired property~(\ref{equ:aux-last}) is satisfied.
	\end{proof}
	
	Lemma~\ref{lem:main-BA} and Theorem~\ref{theo:undecidable} are proved.
\end{proof}


\end{document}